\documentclass[11pt]{amsart}
\usepackage{amssymb,amsthm,amsmath,amstext}
\usepackage{mathrsfs}  
\usepackage{bm}        
\usepackage{multirow}
\usepackage[left=1.25in,top=1.35in,right=1.25in,bottom=1.35in]{geometry}

\usepackage[all]{xy}
\usepackage{rotating}

\theoremstyle{plain}
\newtheorem{theorem}{Theorem}
\newtheorem{prop}[theorem]{Proposition}
\newtheorem{lemma}[theorem]{Lemma}
\newtheorem{corollary}[theorem]{Corollary}

\newtheorem{question}[theorem]{Question}

\theoremstyle{definition}
\newtheorem{definition}[theorem]{Definition}

\theoremstyle{remark}
\newtheorem{remark}[theorem]{Remark}
\newtheorem{example}[theorem]{Example}

\numberwithin{equation}{section}

\newcommand{\cat}[1]{\mathsf{#1}}

\newcommand{\isom}{\cong}

\newcommand{\D}{\cat{D}}
\newcommand{\Db}{\D^{\mathrm{b}}}

\newcommand{\Pic}{\mathrm{Pic}}

\newcommand{\Hom}{\mathrm{Hom}}
\newcommand{\Spec}{\mathrm{Spec}}

\newcommand{\sod}[1]{\langle #1 \rangle}

\newcommand{\id}{{\rm id}}

\newcommand{\GK}{\mathrm{GK}}

\newcommand{\ka}{{\mathcal A}}
\newcommand{\kb}{{\mathcal B}}
\newcommand{\kc}{{\mathcal C}}

\newcommand{\kf}{{\mathcal F}}
\newcommand{\kg}{{\mathcal G}}

\newcommand{\ko}{{\mathcal O}}

\newcommand{\ZZ}{\mathbb{Z}}

\newcommand{\CC}{\mathbb{C}}

\newcommand{\FF}{\mathbb{F}}

\newcommand{\PP}{\mathbb{P}}

\newcommand{\linedef}[1]{\textit{#1}}

\begin{document}

\title[Categorical invariant for conic bundles]
{A categorical invariant for geometrically rational surfaces with a conic bundle structure}

\author{Marcello Bernardara}
 \address{Institut de Math\'ematiques de Toulouse \\ %
 Universit\'e Paul Sabatier \\ %
 118 route de Narbonne \\ %
 31062 Toulouse Cedex 9\\ %
 France}
 \email{marcello.bernardara@math.univ-toulouse.fr}

\author{Sara Durighetto}
\address{Institut de Math\'ematiques de Toulouse \\ %
 Universit\'e Paul Sabatier \\ %
 118 route de Narbonne \\ %
 31062 Toulouse Cedex 9\\ %
 France}
\email{drgsra@unife.it}

\begin{abstract} 
We define a categorical birational invariant for minimal geometrically rational surfaces with a conic bundle
structure over a perfect field via components of a natural semiorthogonal decomposition. Together with the similar
known result on del Pezzo surfaces, this provide a categorical birational invariant for geometrically
rational surfaces.
\end{abstract}

\maketitle

\section{Introduction}

In recent years the study of the derived category of an algebraic variety has been
widely developed. It is clear now that semiorthogonal decompositions can provide
a useful tool in order to detect the geometrical structure of a variety. In particular
its interest focus in finding a birational invariant to be used, for example, to study
the rationality of the variety.

In this context, the first author and M. Bolognesi \cite{bolognesi_bernardara:representability}
introduced the concept of
categorical representability and formulated the following question: is a rational
variety always categorically representable in codimension 2? Analogously, is it
possible to characterize obstruction to rationality via natural components of some
semiorthogonal decomposition which cannot be realized in codimension 2?
On the complex field, for example, if we consider a $V_{14}$ Fano threefold $X$, its
derived category admits a semiorthogonal decomposition with only one nontrivial
component $\ka_X$. For a smooth cubic threefold $Y$ we can also find a decomposition
with only one nontrivial component $\ka_Y$, and
and Kuznetsov showed that $\ka_X$ is equivalent to $\ka_Y$
if $Y$ is the unique cubic threefold birational to $X$ \cite{kuznetsov:v14}. This suggests
that one could consider $\ka_X$ as a birational invariant.

In the case of complex conic bundles $\pi : X \to S$ over a minimal rational surface the situation is quite well-known.
A necessary condition to rationality of $X$ is
that the intermediate Jacobian $J(X)$, as principally polarized Abelian variety, is the
direct sum of the intermediate Jacobian of smooth projective curves. It follows for example that
smooth cubic threefolds are not rational \cite{clemens_griffiths}.
From a categorical point of view it is possible to characterize the rationality of
the conic bundle from the semiorthogonal decomposition of the derived category.
By Kuznetsov \cite{kuznetquadrics} we have
\begin{equation}\label{eq:deco-conic-bd}
\Db(X) = \sod{\Phi \Db(D,\kb),\pi^*\Db(S)},
\end{equation}

where $\kb$ is the sheaf of even parts of Clifford algebras associated to the quadratic
form defining the fibration and $\Phi: \Db(S,\kb) \to \Db(S)$ is a fully faithful functor from
the derived category of $\kb$-algebras over $S$. If $S$ is rational the only nontrivial part
for this semiorthogonal decomposition must then be contained in the component
$\Db(S,\kb)$. If $S$ is minimal, the first author and Bolognesi proved that $X$ is rational if and only
if $\Db(S,\kb)$ has a decomposition whose
components are derived categories of smooth curves or exceptional
objects \cite{bolo_berna:conic}.

All those results holds on the complex field $\CC$, but we want to study
the problem over an arbitrary perfect field $k$.
Auel and the first author worked out the case of del Pezzo surfaces \cite{auel-berna-surf}\footnote{the results
in \cite{auel-berna-surf}
are claimed to hold over general fields, but perfection is required, as we will
show in Remark \ref{rmk:perfect}, to ensure that evey birational map can be factored into
Sarkisov links as in \cite{isko-sarkisov-complete}.}.
Given a minimal del Pezzo surface
$S$ of degree $d$ and Picard rank $1$, a natural subcategory $\ka_S \subset \Db(S)$ is defined by the
orthogonal complement to the structure sheaf. In \cite{auel-berna-surf}, a category $\GK_S$ can be defined,
roughly speaking, as the product of all components of $\ka_S$ which are not representable in dimension $0$,
and it is a birational invariant.
Such Griffiths-Kuznetsov component, where it is defined, is then the suitable
birational invariant to detect the rationality of the given variety.
Our aim is to extend this approach to the other class of geometrically rational minimal surfaces,that is,
conic bundles.

The precise definition of such an invariant is given in Definition \ref{def:GKS}.
Roughly speaking, we define the Griffiths-Kuznetsov component
to be the direct sum of subcategories of $\Db(S)$ which are not representable in dimension $0$. However,
unlikely in the case of del Pezzo surfaces, there is no, to the best of the authors' knowledge, argument
to prove that the (natural) decomposition we choose to define $\GK_S$is unique up to mutations. This motivates the
involved case-by-case definition, and a fundamental part of this work is to prove that $\GK_S$ is 
indeed well-defined. Our main result is the following.

\begin{theorem}\label{thm:main}
Let $k$ be a perfect field and $S$ be a geometrically rational surface birational to a conic bundle over
$k$. The Griffiths-Kuznetsov component $\GK_S$ is well defined and is a birational invariant.
\end{theorem}

Recall the classification of minimal geometrically rational surfaces over an arbitrary field (see, e.g., \cite{hassett-ratsurf}):
minimal conic bundles are one of the two possible classes of such surfaces, namely 
the ones with Picard rank two, the other being del Pezzo surfaces
with Picard rank one. Combining Theorem \ref{thm:main} with the results from \cite{auel-berna-surf},
we obtain the following result.

\begin{theorem}
Let $S$ be a geometrically rational surface over a perfect field $k$.
Then the Griffiths-Kuznetsov component $\GK_S$ is well-defined and it is a birational
invariant.
\end{theorem}

\subsection*{Notations}
Functors of geometric origin between derived categories will be denoted
underived (i.e. $f^*$ instead of $Lf^*$ for the pull-back via a morphism). Given a $k$-algebra
$A$, the notation $\Db(k,A)$ stands for the $k$-linear bounded derived category of coherent
$A$-modules.

\subsection*{Acknowledgments}
The authors are grateful to St\'ephane Lamy for providing and discussing the example in remark
\ref{rmk:perfect}, and to J\'er\'emy Blanc and Michele Bolognesi for fruitful conversations.

\section{Basics on geometrically rational surfaces}

In this section we will introduce some useful and known results.
Let $k$ be a perfect field and $\overline{k}$ an algebraic closure. Let us consider
$S$, a smooth projective geometrically integral surface over $k$. We say that $S$ is
geometrically rational if $\overline{S}:= S \times_k \overline{k}$ is $\overline{k}$-rational.
A field extension $l$ of $k$ is a splitting field for $S$ if $S \times_k l$ is birational to $\PP^2_l$
through a sequence of monoidal transformations centered at closed $l$-points.

A smooth projective surface $S$ is minimal over $k$ if every birational morphism
$\phi : S \to Y$, defined over $k$, to a smooth variety $Y$ is an isomorphism. If $k$ is algebraically
closed, the only minimal rational surfaces are the projective plane and projective
bundles over $\PP^1$. Over a general field, we have the following classification (see, e.g., \cite{hassett-ratsurf}).

\begin{prop}
Let $S$ be a minimal geometrically rational surface over $k$. Then $S$ is one of the following:
\begin{enumerate}
\item $S=\PP^2_k$ is the projective plane, so $\Pic(S)
= \ZZ$, generated by the hyperplane $\ko(1)$;

\item $S \subset \PP^3_k$ is a smooth quadric and $\Pic(S) = \ZZ$, generated by the hyperplane section $\ko(1)$;

\item $S$ is a del Pezzo surface with $\Pic(S) = \ZZ$, generated by the canonical class $\omega_S$;

\item $S$ is a conic bundle $f : S \to C$ over a geometrically rational curve, with $\Pic(S) \simeq \ZZ \oplus \ZZ$.
\end{enumerate}
 
\end{prop}

\subsection{Elementary links}
We recall some elements of the Sarkisov
program which describes the factorization of a birational map between minimal rational surfaces
in elementary links \cite{isko-sarkisov-complete}.
Let $\pi: S \to Y$ be a minimal geometrically rational surface with an extremal
contraction. Then either $Y$ is a point and $S$ is a minimal surface with Picard
rank 1 or $Y$ is a curve and $S$ is a conic bundle with Picard rank 2. If 
$\pi : S \to Y$ and $\pi' : S' \to Y'$ are two extremal contractions, an elementary
link is a birational map $\phi: S \dashrightarrow S'$ of one of the following types:

\begin{enumerate}
 \item[Type I)] There is a commutative diagram
 $$\xymatrix{S \ar[d] & S' \ar[l]_\sigma \ar[d]\\
 Y & Y' \ar[l]_\psi}$$
 where $\phi=\sigma^{-1}$, $\sigma: S' \to S$ is a Mori divisorial elementary contraction and
 $\psi: Y' \to Y$ is a morphism. In this case, $Y = \Spec(k)$,
 $\rho(S)=1$, $S$ is a minimal del Pezzo, and $S' \to Y'$ is a
 conic bundle over a geometrically rational curve.

\medskip

\item[Type II)]  There is a commutative diagram
 $$\xymatrix{S \ar[d] & X \ar[l]_\sigma \ar[r]^\tau& S' \ar[d] \\
 Y & & Y' \ar[ll]_\isom}$$
 where $\phi=\tau\circ \sigma^{-1}$, $\sigma:X \to S$ and $\tau: X \to S'$ are Mori divisorial
 elementary contractions.  In this case, $S$ and $S'$ have the same
 Picard number, and are hence either both del Pezzo surfaces (and $Y$ is a point) or
 both conic bundles (and $Y$ is a geometrically rational curve).

\medskip

\item[Type III)]  There is a commutative diagram
 $$\xymatrix{S \ar[d]\ar[r]^\sigma  & S' \ar[d]\\
 Y \ar[r]^\psi& Y' }$$
where $\phi=\sigma$, $\sigma: S \to S'$ is a Mori divisorial elementary contraction and $\psi: Y \to Y'$ is 
a morphism. Links of type III are nothing but inverse of links of type I.

\medskip

\item[Type IV)] There is a commutative diagram
 $$\xymatrix{S \ar[d]\ar@{-->}[rr]^\phi & & S' \ar[d]\\
 Y \ar[rd]^\psi& & Y' \ar[ld]_{\psi'} \\
 & \Spec(k) & }$$
 where $S \simeq S'$ are isomorphic, $Y$ and $Y'$ are geometrically rational curves and $\psi$ and $\psi'$
 are the structural morphisms. Then both $S$ and $S'$ are conic bundles and the link amounts to a change of conic
 bundle structure on $S$.
\end{enumerate}

Any birational map $\phi: S \dashrightarrow S'$ between minimal geometrically rational surfaces can be
factored through elementary links, and Iskovskikh gives the complete list of all possible such links \cite{isko-sarkisov-complete}.
We note that the Picard rank is invariant under links of type II and IV, while it changes under links of type I and III.
Moreover, if we suppose that $S$ is not rational the list of links of type I (and hence of their inverses of type III)
is very limited: either $S$ is of degree 8 with a point
of degree 2, and $S'$ is of degree 6 and the curve $C$ can be rational (according to $S$ being a quadric or not),
or $S$ is of degree 4, has a rational point and $S'$ is of degree $3$ and $C$ is a rational curve.

\begin{remark}\label{rmk:perfect}
If $k$ is not perfect, than a birational map may not be decomposable in a finite sequence of elementary links
centered at closed points. An example of such map was given in \cite[Rmk. 1.3]{lamy-zimmermann}: if $k=(\ZZ/2\ZZ)[t]$,
the birational map $\phi$ of $\PP^2_k$ given by
$$[x_0:x_1:x_2] \dashrightarrow [x_0x_2:x_1x_2:x_0^2+tx_1^2]$$
has $[\sqrt{t}:1:0]$ as a base point, and such a point is never defined over a separable field extension of $k$.
\end{remark}

\section{Basics on derived categories}

\subsection{Categorical representability}
Using semiorthogonal decompositions, one can define a notion of
\linedef{categorical representability} for triangulated categories. In
the case of smooth projective varieties, this is inspired by the
classical notions of representability of cycles, see
\cite{bolognesi_bernardara:representability}. We refrain here to recall standard
notions of semiorthogonal decompositions, exceptional objects, and mutations, the interested reader can
refer to \cite{auel-berna-frg}. Let us just recall a nonstandard definition of exceptional object.

\begin{definition}
\label{def-except}
Let $A$ be a division (not necessarily central) simple $k$-algebra (i.e., the
center of $A$ could be a field extension of $k$), and $\ka$ a $k$-linear triangulated category.
An object $V$ of
$\ka$ is called \linedef{$A$-exceptional} if
$$
\Hom(V,V) = A \quad \text{and} \quad
\Hom(V,V[r])=0 \quad \text{for} \quad r \neq 0.
$$
An exceptional object in the classical sense of the term
\cite[Def.~3.2]{gorodentsev-moving} is a $k$-exceptional
object.  By \linedef{exceptional} object, we mean $A$-exceptional for
some division $k$-algebra $A$.
\end{definition}

\begin{example}\label{ex:quillen}
Let $A$ be a central simple algebra over $k$ and $X:=SB(A)$ the Severi-Brauer variety associated to
it, and let $n=\dim X$. The Quillen vector bundle $V$ is a rank $n+1$ indecomposable vector bundle whose base change to a splitting
field is $\ko(1)^{\oplus n+1}$, and is in particular an $A$-exceptional object \cite{quillen:higher_K-theory}. 
\end{example}

\begin{definition}
\label{def-rep-for-cat}
A triangulated category $\ka$ is \linedef{representable
in dimension $m$} if it admits a semiorthogonal decomposition
$$
\ka = \langle \ka_1, \ldots, \ka_r \rangle,
$$
and for each $i=1,\ldots,r$ there exists a smooth projective
$k$-variety $Y_i$ with $\dim Y_i \leq m$, such that $\ka_i$ is
equivalent to an admissible subcategory of $\Db(Y_i)$.
\end{definition}

The motivation
for the above definition is the possibility to formulate the following
question:

\begin{question}
Let $X$ be a smooth projective $k$-variety of dimension $n$. Does $X$ rational imply $\Db(X)$ categorically
representable in dimension $n-2$?
\end{question}

In this work, we consider the above question for surfaces, and we are hence interested in characterizing categories which are representable
in dimension 0. This was done in \cite{auel-berna-surf}.

\begin{lemma}\label{lem:dim0dim1}
A triangulated category $\ka$ is representable in dimension $0$ if and only if there exists a semiorthogonal
decomposition $$
\ka = \langle \ka_1, \ldots, \ka_r \rangle,
$$
such that for each $i$, there is a $k$-linear equivalence $\ka_i
\simeq \Db(K_i/k)$ for a separable field extension $K_i/k$
\end{lemma}

\subsection{Conic bundles}
We recall a natural semiorthogonal decomposition of
the derived category of a surface with conic bundle structure, following the work of Kuznetsov \cite{kuznetquadrics} and
its generalization to general fields \cite{auel-berna-bolo}.
Let $S$ be a surface over the field $k$ with a structure of conic bundle $\pi : S \to
C$ over a geometrically rational curve. Such conic bundle is
associated to a quadratic form $q : E \to L$ on a locally free $\ko_C$-module $E$ of rank 3.
Denote by $\ko_{S/C} (1)$ the restriction to $S$ of the line bundle $\ko_{\PP E/C}(1)$, and let $\kb$
be the even Clifford algebra associated to the form $q$, which is a locally free $\ko_C$-algebra whose isomorphism
class is invariant for $\pi:S \to C$.

Under these conditions, we have that $\pi^* : \Db (C) \to \Db (S)$
is fully faithful, and there exist a fully faithful functor $\Phi: \Db (C, \kb) \to \Db (S)$ such
that
$$\Db (S) = \sod{\pi^* \Db(C),\Phi \Db (C, \kb)}.$$

Moreover, since $C$ is a geometrically rational curve, there is a simple $k$-algebra $A$ (trivial if and
only if $C \simeq \PP^1_k$) such that $C=SB(A)$. In particular, there is an $A$-exceptional object $V$,
which is either $\ko(1)$ if $C \simeq \PP^1_k$ or the Quillen bundle $V$ as in example \ref{ex:quillen} if $C$ is not
rational, such that
$$\Db(C)= \sod{\ko_C,V} = \sod{V^*,\ko_C}.$$

It follows that we can refine the semiorthogonal decomposition of $S$ (abusing of notation, setting $V:=\pi^*V$)
as follows:
\begin{equation}\label{eq:the-main-deco}
\Db(S)=\sod{\ko_S,V,\Phi \Db(S,\kb)}.
\end{equation}

Now, let $\overline{\pi}:\overline{S} \to \PP^1_{\overline{k}}$ be the base change of the conic bundle to the algebraic
closure. Such a conic bundle is not necessarily a Hirzebruch surface and can indeed be not minimal, and have
a finite number, say $r$, of singular fibers which are given by two lines meeting in a point. We can pick
one line in each fiber and denote such set of lines by $F_1,\ldots, F_r$. The Picard rank of $\overline{S}$ is then $2+r$, and there is a
semiorthogonal decomposition obtained by considering the $\overline{k}$-minimal model $\overline{S} \to S_0$,
which is a Hirzebruch surface:
\begin{equation}\label{eq:on-the-algclos}
\Db(\overline{S})=\sod{\ko_{\overline{S}},\ko_{\overline{S}}(F),\ko_{\overline{S}}(\Sigma),
\ko_{\overline{S}}(\Sigma+F),\ko_{F_1},\ldots,\ko_{F_r}},
\end{equation}
where $F$ is the general fiber of $\overline{\pi}$, $\Sigma$ is a section of $\overline{\pi}$.

We finally notice that the base change of the semiorthogonal decomposition \eqref{eq:the-main-deco} is 
exactly the semiorthogonal decomposition \eqref{eq:on-the-algclos}: indeed, either $C$ is rational and
we already have $V=\ko_S(F)$, or $C$ is not rational, $V$ has rank $2$ and we have $\overline{V}=
\ko_{\overline{S}}(F)^{\oplus 2}$. The latter generates the same category as $\ko_{\overline{S}}(F)$
since we are considering thick subcategories.

It follows that the base change of $\Phi \Db(C,\kb)$ to $\overline{S}$ is the subcategory

\begin{equation}
\sod{\ko_{\overline{S}}, \ko_{\overline{S}}(F)}^\perp=\sod{\ko_{\overline{S}}(\Sigma),
\ko_{\overline{S}}(\Sigma+F),\ko_{F_1},\ldots,\ko_{F_r}}.
\end{equation}

\section{Links of type I/III and the definition of the Griffiths-Kuznetsov component}
We are going to construct a birational invariant for geometrically rational surfaces with a conic
bundle $\pi: S \to C$ as the collection of subcategories in the semiorthogonal decomposition 
\eqref{eq:the-main-deco} which are not representable in dimension 0. Such an invariant will match the
one constructed in \cite{auel-berna-surf} in the case where $S$ is birational to a minimal del Pezzo
surface.
Hence, we first have to deal with links of type I/III in the non-rational cases to give a proper definition.
Indeed, the subcategory $\Phi\Db(C,\kb)$ can admit semiorthogonal decompositions and even exceptional
object if $S$ is birational to a quadric or to a del Pezzo surface of degree 4.
These cases were already treated in \cite{auel-berna-surf}, and we quickly recall them.

\medskip

Let $S'$ be a minimal non-rational del Pezzo surface of degree 8 with a point of degree 2. Then, $S'$ is
an involution surface in a Severi-Brauer threefold $SB(B)$, and there is an associated even Clifford algebra
$\kc$, which is a simple algebra
whose center is a degree two field extension of $k$, and a semiorthogonal decomposition
$$\Db(S')= \sod{\Db(k),\Db(k,B),\Db(k,\kc)},$$

where the first category is generated by $\ko_{S'}$ and the second one either by $\ko_{S'}(1)$
(in which case $B=0$ and $S'$ is a quadric) or by the restriction of the Quillen bundle of $SB(B)$ to
$S'$ (in the case where $S'$ is not a quadric). It follows that the Griffiths-Kuznetsov component
for $S'$ should be $\GK_{S'}:=\Db(k,\kc)$ if $S'$ is a quadric and $\GK_{S'}:=\Db(k,\kc) \oplus \Db(k,B)$
if $S'$ is not a quadric. In this case, such a category is not shown to be a birational invariant in \cite{auel-berna-surf},
and this is due to the existence of a link of type I, from which follows that the birational class of
$S'$ contains minimal conic bundles of degree 6. Indeed, the blow-up of a degree 2 point $S \to S'$ is a
conic bundle $\pi: S \to C$, with $C$ either rational if $S'$ is a quadric or non-rational if $S'$ is not a quadric.

In \cite[\S B]{auel-berna-surf}, it is proved that the component we want to construct is indeed related to the standard semiorthogonal
decomposition of the conic bundle as follows: writing $C=SB(A)$ we have that $A$ and $B$ are Morita-equivalent
(note that $B$ has order dividing 2 since it has an involution defining $S'$), and that there is a degree 2 extension
$l/k$ and a semiorthogonal decomposition:
$$\Db(C,\kb)=\sod{\Db(l),\Db(k,\kc)}.$$
It follows that the components which are (potentially) not representable in dimension 0 in the standard decomposition
\eqref{eq:the-main-deco} are exactly the ones we considered above for $S'$.

\medskip

Let $S'$ be a minimal del Pezzo surface of degree 4 with a rational point. Note that $S'$ is not rational.
In particular, there is semiorthogonal decomposition
$$\Db(S')=\sod{\ko_{S'},\ka_{S'}},$$
and $\GK_{S'}:=\ka_{S'}$ is expected to be the good candidate for the birational invariant we are looking for.
This case neither was treated in \cite{auel-berna-surf}, since, again, the existence of a link of type I
implies that the birational class of $S'$ contains minimal conic bundles of degree 3. Indeed, the blow-up 
of a rational point $S \to S'$ is a conic bundle $\pi: S \to \PP^1$.

In \cite[A.2]{auel-berna-surf} (see also \cite{auel-berna-bolo}), it is proved that the component we want
to construct is indeed related to the standard semiorthogonal
decomposition of the conic bundle since there is an equivalence $\Db(C,\kb) \simeq \ka_{S'}$.
It follows that the component which is (potentially) not representable in dimension 0 in the standard decomposition
\eqref{eq:the-main-deco} is exactly the one we considered above for $S'$.

\medskip

Now that we have analyzed all the possible links of type I/III between non-rational surfaces, we can give the
definition of the Griffiths-Kuznetsov component of a conic bundle.

\begin{definition}\label{def:GKS}
Let $S$ be a surface with a structure of conic bundle $\pi: S \to C$ over
a geometrically rational smooth curve $C$. If $S$ is minimal, the Griffiths-Kuznetsov
component $\GK_S$ of $S$ is defined as follows:
\begin{enumerate}
 \item if $S$ is rational, $\GK_S= 0$;
 \item if $S = C_1 \times C_2$ where $C_i$ is a geometrically rational curve with associated
Azumaya algebra $A_i$, then $\GK_S$ is the sum of those between $\Db(k,A_1)$, $\Db(k,A_2)$ and $\Db(k,A_1 \otimes A_2)$
which are not equivalent to $\Db(k)$ (equivalently, the algebra is not Brauer-trivial);
\item if $C \simeq \PP^1$ and $S$ is birational to a non-rational quadric with associated even Clifford
algebra $\kc$, then $\GK_S = \Db (k, \kc)$;
\item if $C \simeq \PP^1$, and $S$ is neither rational nor birational to a quadric, $\GK_S = \Db(\PP^1,\kb)$;
\item if $C$ is not rational with associated Azumaya algebra $A$, and $S$ is birational to a quadric
with associated even Clifford algebra $\kc$, then $\GK_S = \Db (k, A)\oplus \Db(k,\kc)$;
\item if $C$ is not rational with associated Azumaya algebra $A$, and $S$ is not birational to a quadric,
then $\GK_S = \Db (k, A)\oplus \Db(C,\kb)$;
\end{enumerate}
If $S$ is not minimal, the Griffiths-Kuznetsov component is $\GK_S = \GK_{S_0}$ for a
minimal model $S \to S_0$. 
\end{definition}

Note that the term component is slightly abused here, since there are cases where $\GK_S$
is not a component of $\Db(S)$ but rather the direct sum of some components. Since we can operate
mutations on semiorthogonal decompositions (and we will indeed do to proof the main theorem),
we cannot in general give any canonical gluing of components contributing to $\GK_S$.

The rest of the paper is dedicated to the proof of Theorem \ref{thm:main}.
Note that we can restrict to minimal models.
If $\pi : S \to C$ and $\pi' : S' \to C'$ be are minimal conic bundle structures
and $\phi: S \dashrightarrow S'$ is a birational morphism, then $\phi$ can be decomposed in
a finite number of links. The invariance of $\GK_S$ under links of type I and III has
been studied above and follows by results from \cite{auel-berna-surf}.
To complete the proof, we will prove the invariance under links of type II
(Theorem \ref{thm:linkII}) and type IV (Corollary \ref{cor:typeIV}).

\section{Links of type II}
Links of type II between conic bundles are the most common birational transformation, that is given by
an elementary transformation along a closed fiber of the conic bundle structure. Let
$\pi:S\to C$ be a conic bundle. To define a link of type II, pick a closed point $x \in S$ of degree $d$, and denote by
$S_x$ the fiber of $\pi$ containing it. Then perform the blow-up of $x$
followed by the subsequent contraction of the fiber $S_x$. This gives a conic bundle $\pi':S' \to C$
and a commutative diagram:
$$\xymatrix{
& Z \ar[dl]_p \ar[dr]^q & \\
S \ar[d]_\pi \ar@{-->}[rr]^\phi & & S' \ar[d]^{\pi'} \\
C \ar[rr]^{\id} & & C
}$$

\begin{theorem}\label{thm:linkII}
In the above setting, we have $\GK_S \simeq \GK_{S'}$. 
\end{theorem}

\begin{proof}
Let $E$ be the exceptional divisor of $p$ and $E'$ the exceptional of $q$. We denote
by $f$ and $f'$ the fibers of $\pi$ and $\pi'$ in $Z$ respectively. Recall
\cite{isko-sarkisov-complete} that
$$
\begin{array}{rl}
p^* (-K_X) &= q^* (-K_Y) + df - 2E'\\
f&=f'\\
E&=df'-E' = p^* (-K_X) - q^* (-K_Y) + E'
  
  \end{array}
  $$
    
Let $V$ be the Quillen bundle on $C$. Since the isomorphism class of $C$ is preserved under
this link, so is the algebra $A$, and therefore the category $\Db(k,A)$.

We are left to prove that the equivalence class of the category $\ka_S:=\Db(C,\kb)$ is also preserved.
Over the algebraic closure, the category $\ka_{\overline{S}} \subset \Db(\overline{S})$ admits the following semiorthogonal
decomposition:
\begin{equation}
\ka_{\overline{S}}=\sod{\ko_{\overline{S}}(\Sigma),
\ko_{\overline{S}}(\Sigma+F),\ko_{F_1},\ldots,\ko_{F_r}},
\end{equation}
where $F_i$ are given a choice of a line in each singular fibers of the fibration $\overline{S} \to \PP^1_{\overline{k}}$,
and $F$ and $\Sigma$ are respectively the fiber and the section of the conic bundle.
Similarly, we have a semiorthogonal decomposition
\begin{equation}
\ka_{\overline{S'}}=\sod{\ko_{\overline{S'}}(\Sigma'),
\ko_{\overline{S'}}(\Sigma'+F'),\ko_{F'_1},\ldots,\ko_{F'_r}}.
\end{equation}

Over $\overline{k}$, we have that $\overline{S}$ is the blow up of $r$ points on a Hirzebruch
surface $\FF_n$, for some $n$, so that $−K_{\overline{S}} = 2\Sigma + (n + 2)F$.
Similarly, $\overline{S'}$ is the blow-up of $r$ points on a Hirzebruch surface $\FF_m$
and one can see that the map $\overline{\phi}$ is obtained by lifting to $\overline{S}$ the composition
of $d$ elementary transformations
on $\FF_n$ along fibers that do not contain the points blown-up by $\overline{S} \to \FF_n$. In 
particular, $m=n-d$ and
$−K_{\overline{S'}} = 2 \Sigma' + (n - d + 2)F' - E'$. It follows that in
our transformation: $\Sigma = \Sigma' - E'$.

Then we have the following equivalence of subcategories of $\Db(\overline{Z})$:
$$p^*\ka_{\overline{S}} \otimes \ko(E') = q^*\ka_{\overline{S'}}.$$
Indeed, first note the singular fibers are preserved under the birational transformation $\phi$, we can choose
$F_i'$ such that $p^*F_i = q^*F_i'$ for $i=1,\ldots,r$. Moreover $\ko_{q^*F_i}$ does not change under tensor
with $\ko(E')$, since the exceptional divisor is not supported on singular fibers.
Secondly, using the above relation $\Sigma = \Sigma' - E'$, it is not difficult to see that
$$\sod{\ko(\Sigma),\ko(\Sigma+F)}\otimes \ko(E') = \sod{\ko(\Sigma'),\ko(\Sigma'+F')}$$
We can now conclude since the autoequivalence $\otimes(E')$ descends to an autoequivalence of
$\Db(Z)$, since $E'$ is defined over $k$. It follows, that $\ka_S$ is equivalent to $\ka_{S'}$
and the proof is complete.
\end{proof}

\section{Links of type IV}
A link of type IV is a birational self-transformation of a minimal surface $S$ exchanging two conic bundle
structures $\pi_i : S \to C_i$, for $i=1,2$. We will than denote by $\ka_i:=\Phi_i\Db(C_i,\kb_i) \subset \Db(S)$ to
keep simple notations. In particular, the birational map $\phi: S \dashrightarrow S$ fits a commutative diagram
 $$\xymatrix{S \ar[d]_{\pi_1}\ar@{-->}[rr]^\phi & & S \ar[d]^{\pi_2}\\
 C_1 \ar[rd]& & C_2 \ar[ld] \\
 & \Spec(k) & }$$
As proved in \cite{isko-sarkisov-complete},
$S$ must have degree 8,4,2 or 1, and the list of birational maps is quite limited.
In this case, we need to prove that the Griffiths-Kuznetsov component is well defined, namely that
it does not depend on the choice of semiorthogonal decomposition given by the different conic
bundle structures. We proceed by a case by case analysis.

\begin{prop}
Let $S$ be a degree $8$ surface and $\phi$ a link of type IV. Then $\GK_S$ is well defined.
\end{prop}

\begin{proof}
In this case, $S = C_1 \times C_2$. The fact that $\GK_S$ is well-defined is proved in \cite[\S C]{auel-berna-surf}
\end{proof}

\begin{prop}
Let $S$ be a surface of degree $4$, $2$, or $1$, and $\phi$ a link of type IV. Then $\GK_S$ is well defined.
\end{prop}

\begin{proof}
For $i=1,2$, consider the semiorthogonal decomposition
$$\Db(S)=\sod{\ko_S,V_i,\ka_i},$$
and recall that $\sod{\overline{V_i}}=\sod{\ko_{\overline{S}}(F_i)}$, for $F_i$ a geometric fiber of $\pi_i$,
since either $C_i$ is rational
and $V_i$ is such a line bundle, or $C_i$ is not rational and $\overline{V_i}=\ko_{\overline{S}}(F_i)^{\oplus 2}$.
Now we proceed by case by case analysis following the possibilities given by \cite{isko-sarkisov-complete}.

\medskip

{\bf Degree 4.} Assume $S$ has degree 4.
We have
$F_1 = -K_S - F_2$, so that $V_1 = V_2^* \otimes \omega^*$. It follows that $\sod{V_1} \simeq \sod{V_2^*} \simeq \sod{V_2}$,
since $A_2^{op}$ and $A_2$ are Brauer equivalent.

Now consider
$$\Db(S)=\sod{V_2^*,\ko_S,\ka_2}=\sod{V_1,\omega_S^*,\ka_2\otimes \omega_S^*}=\sod{\ko_S,\ka_2,V_1},$$
where the first equality is given by the autoequivalence $\otimes \omega_S$
on $\Db(S)$ and the second is the mutation of $\sod{\omega_S^*,\ka_2'}$ to the left with respect to $V_1$. We
can then mutate $\ka_2$ to the right with respect to $V_1$ and obtain then $\ka_1$, so that we have shown that
$\ka_1 \simeq \ka_2$ and we finished the proof.

\medskip

{\bf Degree 2.} Assume $S$ has degree 2.
We have
$F_1 = -2K_S - F_2$, so that $V_1 = V_2^* \otimes (\omega^*)^{\otimes 2}$.
It follows that $\sod{V_1} \simeq \sod{V_2^*} \simeq \sod{V_2}$,
since $A_2^{op}$ and $A_2$ are Brauer equivalent.

Now consider the first conic bundle structure and the semiorthogonal decompositions
$$
\Db(S)= \sod{\ko_S,V_1,\ka_1}=\sod{V_1,\ka_1,\omega_S^*}=\sod{\ka_1',V_1,\omega^*_S}, 
$$
where the first equality is the mutation of $\ko_S$ to the right with respect to its orthogonal
complement and the second one is the mutation of $\ka_1$ to the left with respect to $V_1$,
so that $\ka_1'=^\perp\sod{V_1,\omega_S^*}$ is equivalent to $\ka_1$.

Now consider the second conic bundle structure and the semiorthogonal decompositions
$$
\begin{array}{c}
\Db(S)= \sod{V_2^*,\ko_S,\ka_2}=\sod{V_1,(\omega_S^*)^{\otimes 2},\ka_2 \otimes (\omega_S^*)^{\otimes 2}} =\\
=\sod{\omega_S^*,\ka_2 \otimes (\omega_S^*),V_1} = \sod{\ka_2',\omega_S^*,V_1},
\end{array}
$$
where we first tensor with $(\omega^*_S)^{\otimes 2}$, then mutate $\sod{(\omega_S^*)^{\otimes 2},\ka_2
\otimes (\omega_S^*)^{\otimes 2}}$ to the left with respect to its orthogonal complement, then
mutate $\ka_2 \otimes \omega_S^*$ to the left with respect to $\omega^*_S$. It follows in particular that
$\ka_2'=^\perp\sod{\omega_S^*,V_1}$ is equivalent to $\ka_2$.

Finally, the two semiorthogonal decompositions give the full orthogonality between $V_1$ and $\omega_S^*$,
so that $\sod{\omega_S^*,V_1}=\sod{V_1,\omega_S^*}$. This implies that $\cat{A}_1'=\cat{A}'_2$ and the proof follows.

\medskip

{\bf Degree 1.} Assume $S$ has degree 1.
We have then that $C_i$ are rational and
$F_1 = -4K_S - F_2$. In particular, we only need to prove that the categories $\ka_1$ and $\ka_2$
are equivalent.

Let us consider the first conic bundle structure and the semiorthogonal decompositions:

$$
\begin{array}{c}
\Db(S)=\sod{\ko(-F_1),\ko_S,\ka_1}=\sod{\ko(F_2),(\omega_S^*)^{\otimes 4}, \ka_1 \otimes  (\omega_S^*)^{\otimes 4}}=\\
=\sod{(\omega_S^*)^{\otimes 3},\ka_1\otimes (\omega_S^*)^{\otimes 3},\ko(F_2)}
=\sod{(\omega_S^*)^{\otimes 3},\ko(F_2),\ka_1'},
\end{array}
$$

where we first tensor by $(\omega_S^*)^{\otimes 4}$, then mutate $\sod{(\omega_S^*)^{\otimes 4}, 
\ka_1 \otimes (\omega_S^*)^{\otimes 4}}$ to the left with respect to its orthogonal complement, then mutate $\ka_1
\otimes (\omega_S^*)^{\otimes 3}$ to the right with respect to $\ko(F_2)$.

Now we need to mutate $\ko(F_2)$ to the left with respect to $(\omega_S^*)^{\otimes 3}$. To this end, let us calculate:

$$\Hom^i((\omega_S^*)^{\otimes 3},\ko(F_2)=H^i(S,\ko(3K_S+F_2)).$$

First of all, note that $(3K_S+F_2).F_2 <0$, which implies that $H^0(S,\ko(3K_S+F_2))=0$.
Similarly, by Serre duality we have that $H^2(S,\ko(3K_S+F_2))=H^0(S,\ko(2K_S+F_1))=0$ since
$(2K_S+F_1).F_1<0$.

Finally we are left with $\dim H^1(S,\ko(3K_S+F_2))=-\chi(\ko_S,\ko(3K_S+F_2))$. The latter can be calculated
by Riemann-Roch:
$$\chi(\ko_S,\ko(3K_S+F_2))=\frac{1}{2}(3K_S+F_2)\cdot(2K_S+F_2)+1=-1,$$
since $K_S \cdot F_2=-2$ and $S$ has degree 1. It follows, that there is a unique extension
$$0 \longrightarrow \ko(F_2) \longrightarrow \kf \longrightarrow (\omega_S^*)^{\otimes 3} \longrightarrow 0,$$
which has rank 2 and first Chern class $F_2-3K_S$. Moreover, $\kf$ is the result of the mutation of
$\ko(F_2)$ to the left with respect to $(\omega_S^*)^{\otimes 3}$, so that we end up with the decomposition

\begin{equation}\label{eq:withkf}
\Db(S) = \sod{\kf,(\omega_S^*)^{\otimes 3},\ka_1'}.
\end{equation}

Now consider the second conic bundle structure and the semiorthogonal decompositions:

$$\Db(S)=\sod{\ko_S,\ko(F_2),\ka_2}=\sod{\ko(F_2),\ka_2,\omega_S^*}=\sod{\ko(F_2),\omega_S^*,\ka_2'},$$
where the first equality is the mutation of $\ko_S$ to the right with respect to its right orthogonal,
and $\ka_2'$ is the mutation of $\ka_2$ to the left with respect to $\omega_S^*$ and is therefore
equivalent to $\ka_2$.

We mutate now $\ko(F_2)$ to the right with respect to $\omega_S^*$. A calculation similar to the above one
shows that there is exactly one nontrivial extension
$$0 \longrightarrow \omega_S^* \longrightarrow \kg \longrightarrow \ko(F_2) \longrightarrow 0,$$
which has rank 2 and first Chern class $F_2-K_S$. Moreover, $\kg$ is the result of the mutation of
$\ko(F_2)$ to the right with respect to $\omega_S^*$, and is an exceptional object. Thanks to Gorodentsev \cite{gorodentsev-moving},
exceptional bundles on $S$ are characterized by their rank and their first Chern class. Note that the
$\kf$ and $\kg$ have both rank 2, while the first Chern class of $\kg$ is the first Chern class of
$\kf \otimes \omega_S$. It follows that $\kg \simeq \kf \otimes \omega_S$ is the mutation of
$\ko(F_2)$ to the right with respect to $\omega_S^*$. We hence end up with the decompositions
$$
\begin{array}{c}
\Db(S)=\sod{\omega_S^*,\kf\otimes\omega_S,\ka_2'}=\sod{(\omega_S^*)^{\otimes 2},\kf,\ka_2'\otimes\omega_S^*}=\\
=\sod{\kf,\ka_2'\otimes\omega_S^*,(\omega_S^*)^{\otimes 3}}=\sod{\kf,(\omega_S^*)^{\otimes 3},\ka_2''},
\end{array}
$$
where first we tensor by $\omega_S^*$, then mutate $(\omega_S^*)^{\otimes 2}$ to the right with respect to its right
orthogonal, and $\ka_2''$ is the left mutation of $\ka_2' \otimes (\omega_S^*)^{\otimes 2}$ to the right with respect
to $(\omega_S^*)^{\otimes 3}$ and is therefore equivalent to $\ka_2$. The proof follows then by
comparison with \ref{eq:withkf}.
\end{proof}

\begin{corollary}\label{cor:typeIV}
The Griffiths-Kuznetsov component is well-defined for minimal conic bundles and hence birational
invariant under links of type IV.
\end{corollary}

\end{document}